\documentclass[a4paper,12pt,twoside]{article}

\usepackage{a4,enumerate}
\usepackage[noadjust]{cite}
\usepackage{amssymb,amsmath,amsthm}
\usepackage{comment}
\usepackage{calc}
\usepackage{xcolor}

\swapnumbers
\numberwithin{equation}{section}

\newtheorem{theorem}{Theorem}[section]

\theoremstyle{definition}

\newtheorem{remark}[theorem]{Remark}

\newcommand\restrict{{\vphantom f\mskip1mu\vrule\mskip2mu}}

 \mathchardef\ordinarycolon\mathcode`\:
  \mathcode`\:=\string"8000
  \begingroup \catcode`\:=\active
    \gdef:{\mathrel{\mathop\ordinarycolon}}
  \endgroup

\newcommand\rlim{
\mathchoice{\vcenter{\hbox{${\scriptstyle{+}}$}}}
{\vcenter{\hbox{$\scriptstyle{+}$}}}
{\vcenter{\hbox{$\scriptscriptstyle{+}$}}}
{\vcenter{\hbox{$\scriptscriptstyle{+}$}}}}

\renewcommand\phi{\varphi}
\newcommand\eps{\varepsilon}
\renewcommand\epsilon{\varepsilon}
\renewcommand\theta{\vartheta}

\newcommand\R{\mathbb R}
\newcommand\N{\mathbb N}

\newcommand\cA{\mathcal A}

\newcommand\eul{{\rm e}}

\newcommand\ndash{\rule[.58ex]{\widthof{--}}{0.065ex}} 

\renewcommand\le{\leqslant}

\renewcommand\ge{\geqslant}

\newcommand\slim{\mathop{\rm s\kern.08em\mbox{\rm -}lim}} 
\let\origthanks\thanks
\renewcommand\thanks[1]{\begingroup\let\rlap\relax\origthanks{#1}\endgroup}

\title{A note on absorption semigroups and regularity}

\author{{Amir Manavi\footnote{Goethestr.\,7,
71332 Waiblingen, Germany,
{\tt 
a.man\rlap{\textcolor{white}{lara@aral}}avi@gm\rlap{\textcolor{white}{%
cannstatt}}x.de}}}, 
{Hendrik Vogt\footnote{Fachbereich Mathematik,
Universit\"at Bremen,
Postfach 330 440,
28359 Bremen, Germany,
{\tt 
hendrik.vo\rlap{\textcolor{white}{hugo@egon}}gt@uni-\rlap{\textcolor{white}{%
hannover}}bremen.de}}} and 
{J\"urgen Voigt\footnote{Technische Universit\"at Dresden,
Fachrichtung Mathematik,
01062 Dresden, Germany,
{\tt 
juer\rlap{\textcolor{white}{xxxxx}}gen.vo\rlap{\textcolor{white}{yyyyyyyyyy}}%
igt@tu-dr\rlap{\textcolor{white}{%
zzzzzzzzz}}esden.de}}}
}

\date{}

\begin{document}

\maketitle

\begin{abstract}
It is shown that, in the theory of absorption semigroups,
two possible ways of defining regularity for absorption rates are in 
fact equivalent.
\vspace{8pt}

\noindent
MSC2010: 47D06, 47A55, 47B60
\vspace{2pt}

\noindent
Keywords: Absorption semigroup, regular absorption rate
\end{abstract}

\section*{Introduction}

The notion of `absorption semigroup' was introduced in \cite{voi-86} in the 
context of absorption rates for positive $C_0$-semigroups acting on 
$L_p$-spaces. Let $(\Omega,\mu)$ be a measure space. Given a positive 
$C_0$-semigroup $T=(T(t))_{t\ge0}$ on $L_p(\mu)$ 
with generator $A$, the aim was to 
associate a $C_0$-semigroup with the formal expression $A-V$ as generator, for 
measurable 
$V\colon\Omega\to\R$ as general as possible. 

One of the properties turning out to be of importance was the regularity of 
absorption rates $V\colon\Omega\to[0,\infty)$. On the one hand, it 
was used to compute the generator of the perturbed 
semigroup in the case $p=1$ (cf.~\cite[Corollary~4.3(b)]{voi-86}), and it was 
also used to show form properties of Schr\"odinger operators 
(cf.~\cite[Proposition~5.8(b)]{voi-86}). 
On the other hand, in~\cite[Theorem~3.5]{voi-88} a seemingly stronger version 
of regularity was used to show a dominated convergence theorem for 
absorption semigroups.

It is the purpose of the present note to show that the two versions of 
regularity are in fact equivalent. 

We mention that the concepts sketched above have been treated 
in more generality in~\cite{man-01}, for absorption semigroups 
on Banach function lattices, and 
in~\cite[Chapter~2]{vog-10}, for the case of propagators instead of 
semigroups.

\section{Equivalence of two notions of regularity}
\label{equiv-regular}

Let $(\Omega,\cA,\mu)$ be a measure space, let 
$p\in[1,\infty)$, and let $T$ be a positive $C_0$-semigroup on $L_p(\mu)$, with 
generator $A$. The following notions have been introduced in \cite{voi-86}. 

If $V\colon\Omega\to[0,\infty]$ or $V\colon\Omega\to[-\infty,0]$,  
then $V$ is called \emph{$T\mkern-1.5mu$-admissible} if $V$ is locally 
measurable,
\begin{equation}\label{sg-conv}
T_V(t):=\slim_{n\to\infty} \eul^{t(A-(V\land n)\lor(-n))}
\end{equation}
exists for all $t\ge0$, and $T_V$ thus defined 
constitutes a $C_0$-semigroup. 
We note that then the convergence stated in \eqref{sg-conv} is uniform for 
$t$ in bounded subsets of~$[0,\infty)$, in short expressed as 
$T_V=\slim_{n\to\infty}T_{(V\land n)\lor(-n)}$.
The generator of $T_V$ will be denoted by $A_V$.

For an absorption rate $V\colon\Omega\to[0,\infty]$ we recall two notions 
of regularity:
$V$ is called \emph{$T\mkern-1.5mu$-regular} if $V$ is $T$-admissible and
\[
T_{0,V}:=\slim_{\eps\to0\rlim}T_{\eps V}=T.
\]
$V$ is called \emph{strongly $T\mkern-1.5mu$-regular} if $V$ is 
$T$-admissible and 
\[
(T_V)_{-V}=\slim_{n\to\infty}T_{V- V\land n}=T.
\]
(The first equality in the previous line is always valid; this 
follows from \cite[Lem\-ma~2.4]{voi-86}.)
Regularity was defined in \cite[Definition~2.12]{voi-86}. Strong 
regularity was defined in \cite[Definition~3.1]{voi-88} (under the notion 
`regular'); we adopt the 
terminology `strongly regular' from \cite[Definition~2.3]{ish-94}. The 
following result shows that the above two notions are equivalent.

\begin{theorem}\label{theorem}
(cf.~\cite[Satz~4.1.54]{man-01})
Let $V\colon\Omega\to[0,\infty]$ be $T\mkern-1.5mu$-admissible. Then 
\[
T_{0,V}=(T_V)_{-V}.
\]
In particular,\/ $V$ is $T\mkern-1.5mu$-regular if and only if\/ $V$ is 
strongly $T\mkern-1.5mu$-regular.
\end{theorem}

\begin{proof}
Let $0<\eps\le1$. Then $\eps V=\eps (V\land n)+\eps(V-V\land 
n)\le\eps n+V-V\land n$ for all $n\in\N$, and 
this implies $T_{\eps V}\ge T_{\eps n+V-V\land 
n}=\eul^{-\eps n\mkern1.5mu\cdot\mkern1.5mu}T_{V-V\land n}$; 
recall~\cite[Remark~2.1(a)]{voi-86}. 
Taking first $\eps\to0$ and then $n\to\infty$,
we obtain $T_{0,V}\ge(T_V)_{-V}$.

In order to show the reverse inequality we note that
$T\ge (T_V)_{-V}\ge T_V$ implies 
strong continuity of $(T_V)_{-V}$; so $-V$ is $T_V$-admissible. From 
\cite[Proposition~3.3(b)]{voi-88} we therefore conclude that $V$ is 
strongly $T_V$-regular. Then \cite[Proposition~3.3(a)]{voi-88} implies that 
all positive multiples of $V$ are strongly $T_V$-regular. 
The application of this fact with $\eps V$ instead of $V$ shows that
$V$ is strongly $T_{\eps V}$-regular for all $\eps>0$. 
From~\cite[Proposition~1.3(a)]{voi-88} it follows that 
\[
(T_{\eps V})_{V- V\land n}\le T_{V- V\land n},
\]
and taking $n\to\infty$ in this inequality  we infer $T_{\eps V}\le 
\slim_{n\to\infty}T_{V- V\land n}=(T_V)_{-V}$, 
for all 
$\eps>0$. Letting $\eps\to0$ we obtain $T_{0,V}\le (T_V)_{-V}$.

The last assertion of the theorem is a consequence of the previous equality and 
the very definitions of regularity and strong regularity.
\end{proof}

\begin{remark}
In \cite{voi-86} and \cite{voi-88}, the notions of admissibility and regularity 
have only been 
considered for absorption rates
$V$ taking their values in $\R$. It was observed in \cite{are-bat-93} that it 
is more appropriate to allow for extended real-valued absorption rates $V$.
It was shown in \cite[Proposition~4.4]{are-bat-93} that for a $T$-admissible 
absorption rate $V\colon\Omega\to[0,\infty]$ the set $[V=\infty]$ is a local 
null set. 
Also, it can be shown as in \cite[Proposition~3.3(b)]{voi-88} that for a 
$T$-admissible absorption rate $V\colon\Omega\to[-\infty,0]$ the absorption 
rate 
$-V$ is $T$-admissible, and therefore $[V=-\infty]$ is a local null set.

In view of these observations, for an extended positive or 
negative $T$-admissible
absorption rate $V\colon\Omega\to[-\infty,\infty]$ one can always replace $V$ 
by an equivalent absorption rate taking its 
values in $\R$, and therefore the results of \cite{voi-86}, \cite{voi-88} are 
applicable.
\end{remark}

We conclude the paper by showing that {\ndash} in a certain sense {\ndash} 
the hypothesis of $T$-admissibility in the definition of regularity is 
redundant; see Theorem~\ref{thm-2} and Remark~\ref{rem-2} below. In order to 
make this precise we 
have to extend the definition of $T_V$ to arbitrary locally measurable 
$V\colon\Omega\to[0,\infty]$. Indeed, the limit
\[
T_V(t):=\slim_{n\to\infty}\eul^{t(A-V\land n)}
\]
\sloppy
exists for all $t\ge0$ (recall~\cite[Remark~2.1(c)]{voi-86}), and $T_V$ thus 
defined is a one-parameter semigroup. It
was shown in \cite[Corollary~3.3]{are-bat-93} that then 
$P:=\slim_{t\to0\rlim}T_V(t)$ exists and is a band projection. 
This implies that there exists a decomposition $\mu = \mu_1+\mu_2$ with 
$\mu_1\perp\mu_2$ (`$\perp$' meaning `locally disjoint') such that $R(P) = 
L_p(\mu_1)$. (We refer to \cite{vog-voi-16} for this description of
projection bands in $L_p(\mu)$.)
It follows that the restriction 
of $T_V$ to $L_p(\mu_1)$ exists and is a $C_0$-semigroup.

\fussy
By $(T_V)_{-V}$ we denote the $C_0$-semigroup obtained from 
$T_V\restrict_{L_p(\mu_1\mkern-1.2mu)}$ by applying $-V$ as a 
perturbation. (Note that $-V$ is 
$T_V\restrict_{L_p(\mu_1\mkern-1.2mu)}$-admissible because 
$T_{V\land m-V\land n}\le T$ 
for all $m\ge n\ge0$.) We also extend $(T_V)_{-V}(t)$ to $L_p(\mu)$ as 
$J(T_V)_{-V}(t)P$, for $t>0$, where $J$ denotes the injection of the band 
$R(P)=L_p(\mu_1)$ into $L_p(\mu)$.
\sloppy

\begin{theorem}\label{thm-2}
Let $V\colon\Omega\to[0,\infty]$ be locally measurable, and assume that 
$(T_V)_{-V}(t)=T(t)$ for all $t>0$.

Then $V$ is $T\mkern-1.5mu$-admissible.
\end{theorem}
\fussy

\begin{proof}
The hypothesis implies 
$P = \slim_{t\to0\rlim} (T_V)_{-V}(t) = \slim_{t\to0\rlim} T(t) = I$.
Therefore $T_V$ is a $C_0$-semigroup on 
$L_p(\mu)$.
\end{proof}

\begin{remark}\label{rem-2}
Assuming solely that $V\colon\Omega\to[0,\infty]$ is locally measurable, one 
can show pretty much as in the proof of Theorem~\ref{theorem} that
$(T_V\mkern-1.5mu)_{-V}(t) \mkern-1mu = \mkern-1mu
 \slim_{\eps\to0\rlim} \mkern-2mu T_{\eps V}(t)$
for all $t>0$. We mention that in the course of this proof one also 
obtains $\slim_{t\to0\rlim}T_{\eps V}(t)=P$ for all $\eps\in(0,1]$.
\end{remark}

{
}

\end{document}